\documentclass{amsart}

\usepackage{environ}
\usepackage{graphicx}
\usepackage{setspace}

\NewEnviron{iproof}[1][Proof]{\begin{proof}[\indent\bfseries #1]\BODY\end{proof}}{}

\newtheorem{theorem}{Theorem}[section]
\newtheorem{lemma}[theorem]{Lemma}
\newtheorem{proposition}[theorem]{Propostion}
\newtheorem{corollary}[theorem]{Corollary}

\theoremstyle{definition}
\newtheorem{definition}[theorem]{Definition}

\theoremstyle{remark}

\numberwithin{equation}{section}

\begin{document}
	
	\begin{spacing}{1.5}
		
		\title{relations between Killing, global Jacobi and solenoidal vector fields}
		\author{Changjie Chen}
		\email{changjie.chen@stonybrook.edu}
		\maketitle

	\section{Introduction}
	On Riemannian manifolds, Killing vector fields are one of the most commonly studied types of vector fields. In this article, we will introduce two other kinds of vector fields, which also have some intuitive geometric meanings but are weaker than Killing vector fields. We will investigate the relations between these vector fields.
	
	\begin{definition}
		On a Riemannian manifold, a vector field is called a global Jacobi field if and only if it restricted on every geodesic is a Jacobi field, and a solenoidal field if and only if its divergence is zero.
	\end{definition}
	
	We will quickly show that Killing implies the other two in Section 2, and mainly discuss whether the weaker conditions imply stronger conditions with some additional conditions on the base manifold. The main theorems are,
\end{spacing}
	\begin{spacing}{1.65}
	\end{spacing}
	\begin{spacing}{1.5}
		\noindent
		\textbf{Theorem A.}
		\textit{On compact Riemannian manifolds, among Killing, global Jacobi and solenoidal, global Jacobi or Killing implies the other two.}
	\end{spacing}
	\begin{spacing}{1.65}
	\end{spacing}
		\begin{spacing}{1.5}
		\noindent
		\textbf{Theorem B.}
		\textit{On noncompact Riemannian manifolds, global Jacobi does not imply solenoidal, and global Jacobi plus solenoidal does not imply Killing.}
		\begin{spacing}{1.65}
		\end{spacing}
\end{spacing}
\begin{spacing}{1.5}

	These relationships can also be stated in the language of flows, due to Proposition 2.3. For convenience, throughout this article, we let $ M $ denote a Riemannian manifold, $ X $ a vector field on $ M $, and $ \phi_t $ the flow of $ X $.
	
	\section{Background And An Important Proposition}
	
	\begin{definition}
		Define $ A_X(V):= \nabla _V X $, and $ J_{V,W}(X):= \nabla ^2_{V,W}X-R_{V,X}W ,$ for $ V,  W \in \Gamma (TM) $.
	\end{definition}
	
	\begin{lemma}
		$ J_{(\cdot,\cdot)}(X) $ is symmetric.
	\end{lemma}
	\begin{proof}
		It follows from that for any $ V, W \in \Gamma (TM) ,$
		\begin{equation*}
			J_{V,W}(X)-J_{W,V}(X)=(\nabla ^2_{V,W}(X)-\nabla ^2_{W,V}(X))-R_{V,X}W+R_{W,X}V=0 .
		\end{equation*}
	\end{proof}

	\begin{proposition}
		Killing implies global Jacobi and solenoidal.
	\end{proposition}
	
	\begin{iproof}
		Let $ p $ be an arbitrary point on $ M $. To show that Killing implies global Jacobi, we choose a local coordinate $ (U,x^i) $ near $ p $. Then,
		\begin{equation*}
			\begin{split}
				J_{\partial_j,\partial_k}(X) &= \nabla^2_{\partial_j,\partial_k}X - R_{\partial_j,X}\partial_k \\
				&=\nabla _{\partial_j} ([\partial_k,X])-[\nabla _{\partial_j} \partial_k,X]+\nabla _{[\partial_j,X]} \partial_k \\
				&=(\partial_j \partial_k (X^i)+X^i \partial_l(\Gamma^i_{jk})-\Gamma^l_{jk} \partial_l (X^i)+\Gamma ^i_{kl} \partial_j (X^l)+\Gamma^i_{jl} \partial_k (X^l))\partial_i.
			\end{split}
		\end{equation*}
		Since $ X $ preserves geometry, the coefficient of $ \partial_i $ is $ L_X(\Gamma)^i_{jk}=0 $, where $ \Gamma = (\Gamma ^i_{jk}) $ is the affine tensor of Christoffel symbol of the second kind. And since $ J $ is a (1,3)-tensor, $ J_{V,W}(X)=0, $ for any $ V,W \in \Gamma (TM) $, which is equivalent to that $ X $ is global Jacobi.
		
		To show that Killing implies solenoidal, choose a local geodesic frame $ \{ e_i \} $ centered at $ p $, then $ div X=\sum_{i} \langle \nabla _{e_i} X,e_i \rangle = \sum_{i} \langle A_X (e_i),e_i \rangle =0,$ since X is Killing and then $ A_X $ is skew-symmetric.
	\end{iproof}
	
	Now we introduce the following important proposition, which describes some geometric meanings of these three kinds of vector fields.
	
	\begin{proposition}\
		\begin{enumerate}
			\item $ X $ is Killing $\Leftrightarrow  \phi_t $ preserves metric $ \Leftrightarrow  A_X $ is skew-symmetric
			\item $ X $ is global Jacobi $ \Leftrightarrow  \phi_t $ preserves geodesics and the affine parameters (i.e., X is affine) $ \stackrel{\textcircled{\footnotesize{1}}}\Leftrightarrow J_{(\cdot,\cdot)}(X)=0 $
			\item $ X $ is solenoidal $ \Leftrightarrow  \phi_t $ preserves volume form $ \stackrel{\textcircled{\footnotesize{2}}}\Leftrightarrow $ div$ X=0 $
		\end{enumerate}
	\end{proposition}
	
		We only prove part of this theorem, because the rest is easy or already well known.
		\begin{iproof}
			$ \stackrel{\textcircled{\footnotesize{1}}}\Leftarrow: $ Let $ \gamma_0 $ be an arbitrary geodesic, $ \gamma_t = \phi_t (\gamma) $, and $ T=\dfrac{\partial}{\partial s}\gamma_t(s) $, then $ [X,T]=0 $, and from $ J_{T,T}(X)=0 $, we obtain that $ [X,\nabla_T T]=0 $. Let $ \phi'_t $ be the flow of $ \nabla_T T $, then $ \phi'_t\phi_t(\gamma_0(s)) = \phi_t\phi'_t(\gamma_0(s)) = \phi_t (\gamma_0(s)) $, since $\nabla_T T|_{t=0}=0 $. Hence, $ \phi_t $ preserves geodesics and the affine parameters.
			
			$ \textcircled{\footnotesize{2}}: $ By Cartan's magic formula, $ L_X vol =di_X vol +i_Xd vol =(div X) vol $,\\
			from which $\textcircled{\footnotesize{2}}$ follows.
		\end{iproof}
	
		Now, we can restate our main theorems in the language of flows.
		
	\end{spacing}
	\begin{spacing}{1.65}
	\end{spacing}
	\begin{spacing}{1.5}
		\noindent
		\textbf{Theorem A'.}
		\textit{On compact Riemannian manifolds, among metric, geodesics (with the affine parameters) and volume form, $ \phi_t $ preserves metric or geodesics (with the affine parameters) implies it preserves the other two.}
	\end{spacing}
	\begin{spacing}{1.65}
	\end{spacing}
	\begin{spacing}{1.5}
		\noindent
		\textbf{Theorem B'.}
		\textit{On noncompact Riemannian manifolds, if $ \phi_t $ preserves geodesics (with the affine parameters), it may not preserve volume form, and if $ \phi_t $ preserves geodesics (with the affine parameters) and volume form, it may not preserve metric.}
		\begin{spacing}{1.65}
		\end{spacing}
	\end{spacing}
	\begin{spacing}{1.5}
	
	\section{On Compact Riemannian Manifolds}
	A compact Riemannian manifold enables us to integrate over the maifold, so some weaker conditions may imply stronger conditions.
	
	\begin{lemma}
		On any Riemannian manifolds, global Jacobi implies that divergence is constant..
	\end{lemma}
	\begin{iproof}
		Choose an arbitrary point $ p \in M $, and a local geodesic frame $ \{e_i\} $ centered at $ p $. Since $ X $ is global Jacobi, from Lemma 2.2, we have, $J_{e_i,e_j}(X)=0, $ for all $ i,j$, then
		\begin{equation*}
			\begin{split}
				0 &= \langle J_{e_i,e_j}(X),e_j \rangle\\
				&=\langle \nabla ^2_{e_i,e_j}X-R_{e_i,X}e_j,e_j \rangle\\
				&= \langle \nabla _{e_i} \nabla _{e_j} X,e_j \rangle\\
				&= e_i \langle \nabla _{e_j} X,e_j \rangle,
			\end{split}
		\end{equation*}
		which implies that div$ X=\sum_j \langle \nabla _{e_j} X,e_j \rangle $ is constant.
	\end{iproof}
	
	\begin{corollary}
		On compact orientable Riemannian manifolds, global Jacobi implies solenoidal.
	\end{corollary}
	\begin{iproof}
		Let $ C= $ div$ X $ and by Stokes' Theorem,
		\begin{equation*}
			\begin{split}
				0 &= \int_M  div X dvol\\
				&= Cvol(M),
			\end{split}
		\end{equation*}
		then div$ X \equiv C=0 $.
	\end{iproof}
	
	\begin{theorem}
		On compact orientable Riemannian manifolds, global Jacobi implies Killing.
	\end{theorem}
	\begin{iproof}
		On the compact and orientable Riemannian manifold $ M $, by Stokes' Theorem and some calculation, we have the following formula (see [1]),
		\begin{equation*}
			\begin{split}
				0 &= \int_M ((X^j_{,k}X^k)_{,j}-(X^j_{,j}X^k)_{,k} + \dfrac{1}{2}g^{jk}(X_i X^i)_{,jk}) d vol\\
				&= \int_M (X_i(g^{jk} X^i_{,jk}-R^i_l X^l) + \dfrac{1}{2}(X^{j,k}+X^{k,j})(X_{j,k}+X_{k,j})-X^j_{,j}X^k_{,k})d vol,
			\end{split}
		\end{equation*}
		where $ X_i=X^j g_{ij}, X^{j,k}=X^j_{,l}g^{kl} $ are canonical type changes on Riemannian manifolds. We rewrite this formula as
		\begin{equation}
		\begin{split}
		\int_M (X_i g^{jk} J^i_{jk}(X) + \dfrac{1}{2} (X^{j,k}+X^{k,j})(X_{j,k}+X_{k,j}) - (div X)^2)d vol =0,
		\end{split}
		\end{equation}
		where $ J_{\partial_j,\partial_k}(X)=J^i_{jk}(X)\partial_i $. Since X is global Jacobi, by Lemma 2.2 and Corollary 3.2, (3.1) implies that
		\begin{equation}
		\int_M (X^{j,k}+X^{k,j})(X_{j,k}+X_{k,j}) d vol=0.
		\end{equation}
		Next we prove that the integrand is non-negative and thus X is Killing. By direct calculation,
		\begin{equation*}
			\begin{split}
				X^{j,k} &= X^j_{,l}g^{lk} \\
				&= (\partial_l (X^j)+X^m\Gamma^j_{ml})g^{lk} \\
				&= (\partial_l (X_m)g^{mj} + X_m\partial_l(g^{mj}) + X^m\Gamma ^j_{ml})g^{lk} \\
				&= X_{m,l}g^{mj}g^{lk}+X^p (g^{mj} g_{np} \Gamma ^n_{ml} + \Gamma ^j_{pl} + g_{np} \partial_l (g^{nj})) g^{lk} \\
				&= X_{m,l}g^{mj}g^{lk} + X^p (g^{mj} \partial_l (g_{mp}) + g_{np} \partial_l (g^{nj})) g^{lk} \\
				&= X_{m,l}g^{mj}g^{lk}.
			\end{split}
		\end{equation*}
		Thus,
		\begin{equation*}
			(X^{j,k}+X^{k,j})(X_{j,k}+X_{k,j}) = (X_{j,k}+X_{k,j})(X_{m,l}+X_{l,m})g^{mj}g^{lk}.
		\end{equation*}
		Since $ G^{-1}=(g^{ij}) $ is positive definite, (3.2) implies $ X_{j,k}+X_{k,j} = 0, $ for all $ j,k $, i.e., X is Killing.
	\end{iproof}
	
	\begin{corollary}
		On compact Riemannian manifolds, global Jacobi implies Killing.
	\end{corollary}
	\begin{iproof}
		Here we only need to concern non-orientable Riemannian manifolds. In this case, take the orientable double covering space with lifted metric and then we apply Corollary 3.2 and Theorem 3.3.
	\end{iproof}
		
	\section{On Noncompact Riemannian Manifolds}
	Based on our disscussion in Section 2, because a general noncompact Riemannian manifold may not have a finite volume, the flow of a global Jacobi field may act on the manifold like a  dilation, which may not preserve volume form. Moreover, even it also preserves volume form, it may not preserve all the geometry of the base manifold. We conclude this in the following theorem,
	
	\begin{theorem}
		On noncompact Riemannian manifolds, global Jacobi does not imply solenoidal, and global Jacobi plus solenoidal does not imply Killing.
	\end{theorem}
	\begin{iproof}[Example]
		Lemma 3.1 tells us that divergence of a global Jacobi field is constant (0 if the base manifold is compact), so we can easily construct a simple example of a global Jacobi field on a noncompact Riemannian manifold with divergence a nonzero constant: $ X=x\dfrac{\partial}{\partial x}+y\dfrac{\partial}{\partial y} $ on $ E^2 $, which corresponds to standard radial dilations of Euclidean spaces. A little modification to this gives an example of that on noncompact Riemannian manifolds, even a vector field is global Jacobi and solenoidal, it may not be Killing: $ X=x\dfrac{\partial}{\partial x}-y\dfrac{\partial}{\partial y} $ on $ E^2 $.
	\end{iproof}

	\bibliographystyle{amsplain}

\begin{thebibliography}{10}
		
		\bibitem {A} K. Yano, \textit{On Harmonic and Killing Vector Field}, Ann. of Math.
		\textbf{55} (1952), 38-45.
		
	\end{thebibliography}

	\end{spacing}
\end{document}